\documentclass{article}
\usepackage{amsmath}
\usepackage{amsfonts}
\usepackage{amssymb}
\usepackage{amsthm}
\usepackage{mathtools}
\usepackage{tikz}
\usepackage{tikz-cd}
\usepackage{enumerate}
\usetikzlibrary{patterns}

\newcommand{\N}{\ensuremath{\mathbb{N}}}

\newcommand{\F}{\ensuremath{\mathbb{F}}}

\newcommand{\nsgp}[1][]{\ensuremath{\triangleleft_{#1}}}

\newtheorem{theorem}{Theorem}[]

\newtheorem{lem}[theorem]{Lemma}

\theoremstyle{definition}
\newtheorem{defn}[theorem]{Definition}

\theoremstyle{remark}
\newtheorem*{rmk}{Remark}

\theoremstyle{plain}

\newcounter{introthmcount}
\theoremstyle{definition}

\title{A criterion for residual $p$-finiteness of arbitrary graphs of finite $p$-groups}
\author{Gareth Wilkes}
\begin{document}
\maketitle
\begin{abstract}
We establish conditions under which the fundamental group of a graph of finite $p$-groups is necessarily residually $p$-finite. The technique of proof is independent of previously established results of this type, and the result is also valid for infinite graphs of groups.
\end{abstract}
\subsection*{Introduction}
Residual properties of graphs of groups have long been a subject of study, and have for instance been particularly important in relation to 3-manifold groups \cite{Hempel87, WZ10, AF13, Wilkes16b}. Any study of such properties almost inevitably involves a reduction to the study of graphs of finite groups. The fundamental group of any finite graph of finite groups is well-known to be residually finite \cite[Proposition II.2.6.11]{SerreTrees}. However the situation for properties of {\em residual $p$-finiteness} is rather more subtle. Throughout the paper, let $p$ be a prime.
\begin{defn}
A group $G$ is {\em residually $p$-finite} if for any $g\in G\smallsetminus 1$ there exists a homomorphism from $G$ to a finite $p$-group whose kernel does not contain $g$; or equivalently, if there exists a normal subgroup of $G$ with index a power of $p$ which does not contain $g$. 
\end{defn}
It is emphatically {\em not} the case that the fundamental group of any finite graph of finite $p$-groups is residually $p$-finite, and one must impose a condition on the graph of groups for this to be the case. 

Higman \cite{Hig64} studied this problem in the case of amalgamated free products, and proved that the existence of chief series for the $p$-groups satisfying a certain compatibility condition (see condition I of Theorem \ref{MainThm}) is necessary and sufficient for an amalgamated free product of $p$-groups to be residually $p$-finite. A similar criterion for HNN extensions was proved by Chatzidakis \cite{chat94}.

In both of these papers the strategy of the proof is to use wreath products in the following manner. For $G=A_1\ast_B A_2$ or $G=A_1\ast_B$, iterated wreath products are used to construct an explicit finite $p$-group $P$ and a map $G\to P$ whose restriction to the $A_i$ is an injection. This map then has free kernel which is a normal subgroup of index a power of $p$. Since free groups are residually $p$-finite, this implies (see Lemma \ref{PassToSubgroups}) that the original fundamental group of the graph of groups is also residually $p$.

Since finite graphs of groups can be constructed step-by-step as iterated amalgamated free products and HNN extensions, these two papers together could be applied to prove that a condition on chief series implies that a finite graph of $p$-groups is residually $p$-finite. However such arguments cannot access infinite graphs of groups.

The purpose of this note is to give a new proof of a chief series condition for residual $p$-finiteness. There are no wreath products, and we will analyse all graphs of groups directly rather than building them up from one-edge graphs of groups (i.e.\ amalgams and HNN extensions). In particular, our proof is also valid for infinite graphs of groups. The criterion studied is given in Theorem \ref{MainThm}.

The scheme of the proof is to use the language of Bass--Serre theory to give a reformulation of our criterion in terms of an action on the tree dual to the graph of groups. In this formulation we can pass to an index $p$ normal subgroup given by a graph of groups which is `simpler' in a certain sense. This process concludes with a free group, which is well-known to be residually $p$-finite, thus proving the theorem. 

\subsection*{Graphs of groups and trees}
For this paper we will use the notions of graphs and graphs of groups as given by Serre \cite[Section I.5.3]{SerreTrees}. We recall these notions and set up notation as follows. A graph $X=VX\sqcup EX$ consists of a set $VX$ of vertices and a set $EX$ of edges. Each edge $y$ has an opposite edge $\bar y$, and has endpoints $o(y)$ and $t(y)$ with $o(\bar y) = t(y)$ and $\bar{ \bar y} =y$. 

A graph of groups $(X, G_\bullet)$ consists of the following data:
\begin{itemize}
\item a connected graph $X$;
\item a group $G_x$ for each $x\in VX \cup EX$, with $G_y = G_{\bar y}$ for $y\in EX$; and
\item monomorphisms $f_y\colon G_y\hookrightarrow G_{t(y)}$ for all $y\in EX$.
\end{itemize}
We fix a maximal subtree $T$ of $X$, which exists by Zorn's lemma if $X$ is infinite. Choose also an orientation $E^{+\!}X$ of $X$---that is, a subset $E^{+\!}X\subseteq EX$ such that for all $y\in EX$ exactly one of $y$ and $\bar y$ is in $E^{+\!}X$. Define a function $\epsilon \colon EX \to \{0,1\}$ by
\[\epsilon(y) = \begin{cases} 0 & \text{ if }y\in E^{+\!}X\\ 1 & \text{ if }y\notin E^{+\!}X\end{cases}\]
The fundamental group $G=\pi_1(X,G_\bullet)$ is then defined to be the group obtained from the free product of the $G_x$ (for $x\in VX\cup EX$) and the free group generated by letters $s_y$ (for $y\in EX$), subject to the following relations:
\begin{itemize}
\item $g = s_y^{1-\epsilon(y)} f_y(g) s_y^{\epsilon(y)-1}$ for $y\in EX$ and $g\in G_y$;
\item $s_y = 1$ for all $y\in ET$, and $s_{\bar y} = s_y$ for all $y\in EX$.
\end{itemize}
All the groups $G_x$ inject into $G$ under this construction, and we identify them with their images in $G$. Note that for an edge $y$ of $X$, the group $G_y$ is not necessarily contained in $G_{t(y)}$, but in a conjugate of it; and note that the map $f_y$ is equal to the composition
\begin{equation}\label{fevsinjection}
\begin{tikzcd}
f_y = \big(G_y \ar{r}{\subseteq} & s_y^{1-\epsilon(y)} G_{t(y)} s_y^{\epsilon(y)-1} \ar{r}& G_{t(y)} \big)  
\end{tikzcd}
\end{equation}
where the final map is left conjugation by $s_y^{\epsilon(y)-1}$.

The Bass--Serre tree of $G$ dual to $(X,G_\bullet)$ is the tree $\widetilde X$ with vertex and edge sets
\[V\widetilde X = \bigsqcup_{x\in VX} G/G_x, \quad E\widetilde X = \bigsqcup_{y\in EX} G/G_y \]
For $x\in X$ define $\tilde x$ to be the coset $1\cdot G_x$ viewed as an element of $\widetilde X$. The adjacency maps in $\widetilde X$ are
\[o(g\tilde y) = gs_y^{-\epsilon(y)}\widetilde{o(y)},\quad t(g\tilde y) = gs_y^{1-\epsilon(y)}\widetilde{t(y)} \]
There is a natural (left-)action of $G$ on $\widetilde X$ with quotient graph $X$ and with point stabilisers
\[ G(g\tilde x) \coloneqq {\rm stab}_G(g\tilde x) = gG_x g^{-1} \] 
Conversely \cite[Section I.5.4]{SerreTrees}, an action of $G$ on a tree $\widetilde X$ gives rise to a graph of groups $(X,G_\bullet)$ whose Bass--Serre tree is $G$-isomorphic to $\widetilde X$.
\subsection*{Results}
For the proof of the main theorem we will make use of the following standard fact. We include a proof for completeness. Note that there is no requirement that $G$ be finitely generated. The notation `$H\nsgp[p] G$' means `$H$ is a normal subgroup of $G$ with index a power of $p$'.
\begin{lem}\label{PassToSubgroups}
Let $H\nsgp[p] G$. If $H$ is residually $p$-finite then $G$ is residually $p$-finite.
\end{lem}
\begin{proof}
Let $g\in G$. If $g\notin H$ then there is nothing to prove. If $g\in H$ then by assumption there is $U\nsgp[p] H$ such that $g\notin U$. Consider
\[ V=\bigcap_{g\in G} gUg^{-1}. \]
Since $U$ is normal in $H$ and $H$ has finite index in $G$, there are only finitely many subgroups in this intersection. All are normal in $H$ of $p$-power index, so the intersection $V$ also has $p$-power index in $H$, and hence in $G$. By construction $V$ is normal in $G$ and $g\notin V$. This completes the proof.
\end{proof}

We proceed now to the main theorem. The criterion for residual $p$-finiteness is stated in terms of chief series. 
\begin{defn}
A {\em chief series} for a finite $p$-group $P$ is a sequence 
\[P=P^{(0)} \geq P^{(1)}\geq \cdots \geq P^{(k)}\geq\cdots \]
of normal subgroups of $P$ such that each successive quotient \[\gamma^{(k)}(P)= P^{(k)}/P^{(k+1)}\] is either trivial or of order $p$ and such that $P^{(n)}=1$ for some $n$. The {\em length} of the chief series is the smallest $n$ such that $P^{(n)}=1$.
\end{defn} 
\begin{rmk}
This differs slightly from the usual definition of a chief series in that the sequence does not terminate. This is a purely formal difference enabling us to state the next theorem in its greatest generality.
\end{rmk}
\begin{theorem}\label{MainThm}
Let $(X,G_\bullet)$ be a graph of finite $p$-groups with fundamental group $G=\pi_1(X,G_\bullet)$.  Suppose there is a chief series $(G_x^{(k)})_{k\geq 0}$ of $G_x$ for each $x\in X$, such that the following two properties hold.
\begin{enumerate}[(I)]
\item\label{CondI} For all $k\in\N$ and all $y\in EX$, we have $f_y(G^{(k)}_y) = f_y(G_y) \cap G^{(k)}_{t(y)}$.
\item\label{CondII} For each $k$ there exists a family of injections $\phi^{(k)}_x\colon \gamma^{(k)}(G_x)\hookrightarrow \F_p$ (for $x\in X$) such that the diagrams
\[\begin{tikzcd}
\gamma^{(k)}(G_y)\ar{rr}{f_y} \ar[hook]{dr} & & \gamma^{(k)}(G_{t(y)})\ar[hook]{dl}\\
&\F_p & 
\end{tikzcd}\]
commute for all $y\in EX$. 
\end{enumerate}
Then $G$ is residually $p$-finite.
\end{theorem}
\begin{rmk}
The converse to Theorem \ref{MainThm} holds if the graph $X$ is finite. In this case, if $G$ is residually $p$-finite then there is a finite $p$-group $P$ and a map $\Phi\colon G\to P$ restricting to an injection on all the (finitely many) $G_x$. Taking intersections of $\Phi(G_x)$ with a chief series $(P^{(k)})_{0\leq k\leq n}$ of $P$ yields chief series of the $G_x$ satisfying the conditions of the theorem.
\end{rmk}
\begin{rmk}
If $X$ is finite, then the conditions of the theorem are also sufficient for $G$ to be conjugacy $p$-separable---this is equivalent to residual $p$-finiteness by \cite[Theorem 4.2]{Toinet13}.
\end{rmk}
\begin{rmk}
We note that in the case that $X$ is a tree, condition \ref{CondII} follows from condition \ref{CondI}: one may choose $\phi^{(k)}_x$ arbitrarily at one point of each connected component of the graph 
$Y_k = \{x\in X\mid \gamma^{(k)}(G_x)\neq 1\}$, whereupon the maps $\phi^{(k)}_x$ for the remaining $x$ in that component of $Y_k$ may be uniquely defined by forcing condition \ref{CondII} to hold.

If $X$ is not a tree, one may still define the $\phi^{(k)}_x$ consistently on a maximal forest $T$ of $Y_k$. For the remaining edges $y\in Y_k\cap E^{+\!}X$, one may again define $\phi^{(k)}_y$ so that condition \ref{CondII} holds. The only remaining cases of condition II that must be satisfied are for the edges $\bar y_0$ for $y_0\in E^{+\!}X\smallsetminus T$. Take an edge path $y_1, \ldots, y_m$ in $Y_k$ from $o(y_0)$ to $t(y_0)$. Then condition \ref{CondII} is easily seen to be satisfied if and only if the composite map
\[\begin{tikzcd}\gamma^{(k)}(G_{y_0}) \ar{r}{f_{\bar y_0}} & \gamma^{(k)}(G_{o(y_0)}) \ar{r}{f_{\bar y_1}^{-1}} & \gamma^{(k)}(G_{y_1}) \ar{r}{f_{y_1}} & \gamma^{(k)}(G_{t(y_1)})\ar{r} & \ldots \\
&\makebox[5em]{\ldots} \ar{r} & \gamma^{(k)}(G_{t(y_0)}) \ar{r}{f_{y_0}^{-1}} & \gamma^{(k)}(G_{y_0})  \end{tikzcd}\]
is the identity. This condition may be seen as the analogue in our context for the condition ($\ast\ast$) on HNN extensions given in \cite{chat94}.
\end{rmk}
The proof of Theorem \ref{MainThm} proceeds most smoothly if we translate conditions \ref{CondI} and \ref{CondII} of Theorem \ref{MainThm} into the language of the Bass--Serre tree dual to the graph of groups $(X,G_\bullet)$.
\begin{lem}\label{LemTrees}
Let $G$ be a group. Let $G$ act on a Bass--Serre tree $\widetilde X$ dual to a graph of finite $p$-groups $(X,G_\bullet)$. Then $(X,G_\bullet)$ satisfies the conditions of Theorem \ref{MainThm} if and only if there exists a chief series $(G(z)^{(k)})_{k\geq 0}$ for each stabiliser $G(z)$ of $z\in\widetilde X$ such that the following conditions hold.
\begin{enumerate}[(I\,$'$)]
\item For all $z\in E\widetilde X$, we have $G(z)^{(k)} = G(z)\cap G(t(z))^{(k)}$ and for each $z\in \widetilde X$ and each $g\in G$, we have $gG(z)^{(k)}g^{-1}= G(g\cdot z)^{(k)}$
\item For each $k$ there exists a family of injections $\psi^{(k)}_z\colon \gamma^{(k)}(G(z))\hookrightarrow \F_p$ for $z\in \widetilde X$ such that the diagrams
\[\begin{tikzcd}
\gamma^{(k)}(G(z))\ar{rr} \ar[hook]{dr} & & \gamma^{(k)}(G(t(z)))\ar[hook]{dl}\\
&\F_p & 
\end{tikzcd}\]
commute for all $z\in EX$, and such that for all $z\in \widetilde X$ and all $g\in G$ the diagram
\[\begin{tikzcd}
\gamma^{(k)}(G(z))\ar{rr}{\zeta_g} \ar[hook]{dr} & & \gamma^{(k)}(G(g\cdot z))\ar[hook]{dl}\\
&\F_p & 
\end{tikzcd}\]
commutes where $\zeta_g$ denotes left conjugation by $g$.
\end{enumerate}
\end{lem}
\begin{proof}
Suppose we have chief series $(G_x^{(k)})_{k \geq 0}$ for the graph of groups $(X,G_\bullet)$ satisfying conditions \ref{CondI} and \ref{CondII} of Theorem \ref{MainThm}. For $g\widetilde x\in \widetilde X$ define a chief series\[ G(g\widetilde x)^{(k)} = gG_x^{(k)}g^{-1}\]
of $G(g\tilde x)$. This is well-defined (that is, it is invariant under replacing $g$ by $gh$ for $h\in G_x$) because $G_x^{(k)}$ is normal in $G_x$. Further define the map $\psi^{(k)}_{g\tilde x}$ to be the composition
\[\begin{tikzcd} \gamma^{(k)}(G(g\tilde x)) \ar{r}{\zeta_{g^{-1}}} & \gamma^{(k)}(G(\tilde x)) = \gamma^{(k)}(G_x) \ar{r}{\phi^{(k)}_x} \ar{r} & \F_p\end{tikzcd}\]
This map is again well-defined under replacing $g$ by $gh$ for $h\in G_x$ because the conjugation action of $G_x$ on itself induces the identity map on the $\gamma^{(k)}(G_x)$---a $p$-group cannot induce a non-trivial automorphism of either the trivial group or of a cyclic group of order $p$.

The parts of conditions I$'$ and II$'$ concerning invariance under $G$-conjugation hold by construction. The conditions on edges follow from conditions I and II of Theorem \ref{MainThm} together with the $G$-conjugation invariance, by recalling from \eqref{fevsinjection} that the inclusion of an edge stabiliser into a vertex stabiliser is, up to a $G$-conjugacy, equal to the map $f_y$ followed by a conjugation by $s_y^{\epsilon(y)-1}$.

Conversely, given chief series for the point stabilisers $G(z)$ satisfying I$'$ and II$'$, we may define 
\[G_x^{(k)} = G(\tilde x)^{(k)}, \quad \phi^{(k)}_x = \psi^{(k)}_{\tilde x} \colon \gamma^{(k)}(G_x)\to \F_p \]
for $x\in X$. Conditions I and II now follow from conditions I$'$ and II$'$ via the expression \eqref{fevsinjection} of the maps $f_y$ as a composition of an inclusion of edge stabilisers and a conjugacy.
\end{proof}
\begin{proof}[Proof of Theorem \ref{MainThm}]
Suppose first that the chief series $(G^{(k)})_{k\geq 0}$ all have length at most $N$ for some $N$---this is automatic if $X$ is finite. We prove the theorem by induction on $N$. If $N=0$ then $G$ is free, hence is residually $p$-finite. Suppose $N>0$. The maps $\phi^{(0)}_x$ in condition \ref{CondII} define, by the universal property of the fundamental group of a graph of groups, a homomorphism $\Phi\colon G\to \F_p$ whose restriction to each $G_x$ is the composite
\[\begin{tikzcd} G_x\ar{r} &G_x / G^{(1)}_x = \gamma^{(0)}(G_x) \ar[hook]{r}{\phi^{(0)}_x} & \F_p\end{tikzcd}\]
Let $H=\ker\Phi$. The group $G$ acts on its Bass--Serre tree $\widetilde X$ as in Lemma \ref{LemTrees}. Consider the action of $H$ on $\widetilde X$. The point stabilisers $H(z)$ for $z\in \widetilde X$ are by construction the groups $G(z)^{(1)}$. The chief series $H(z)^{(k)}=G(z)^{(k+1)}$ now automatically satisfy conditions I$'$ and II$'$, and all have length at most $N-1$. Hence by Lemma \ref{LemTrees} the graph of groups decomposition of $H$ dual to its action on $\widetilde X$ is equipped with chief series of length  at most $N-1$ satisfying conditions \ref{CondI} and \ref{CondII}. Therefore by induction $H$ is residually $p$-finite. Since $H$ is a normal subgroup of index $p$ in $G$, it follows from Lemma \ref{PassToSubgroups} that $G$ is also residually $p$-finite.

Now move to the general case. Let $g\in G\smallsetminus\{1\}$. In the graph of groups $(X,G_\bullet)$, the element $g$ is equal to a reduced word \cite[Section I.5.2]{SerreTrees} which is supported on some finite subgraph $Z$ of $X$. Let $N$ be such that the chief series $(G_z^{(k)})_{k\geq 0}$ of $G_z$ has length at most $N$ for all $z\in Z$. We may take the quotient of each $G_x$ by $G_x^{(N)}$ to obtain a new graph of groups $(X,G_\bullet / G^{(N)}_\bullet)$. There is a natural map of fundamental groups 
\[\Phi\colon G=\pi_1(X,G_\bullet) \to \pi_1(X,G_\bullet / G^{(N)}_\bullet)\eqqcolon G'\]
Then $\Phi(g)$ is non-trivial in $G'$, for it is given by a reduced word---the same reduced word as in $G$, since $G^{(N)}_z=1$ for all $z\in Z$. But $G'$ is residually $p$-finite by the first part of the theorem, since the chief series for all $G_\bullet / G^{(N)}_\bullet$ have length at most $N$. Therefore there is some map $\Psi\colon G'\to P$ for a finite $p$-group $P$ such that $\Psi\Phi(g)\neq 1$. This attests that $G$ is residually $p$-finite.
\end{proof}
\subsection*{Acknowledgements}
The author was supported by a Junior Research Fellowship from Clare College, Cambridge.
\bibliographystyle{alpha}
\bibliography{EmbGG.bib}

\begin{thebibliography}{Hem87}

\bibitem[AF13]{AF13}
Matthias Aschenbrenner and Stefan Friedl.
\newblock {3-manifold groups are virtually residually $p$}.
\newblock {\em {Memoirs of the American Mathematical Society}}, 225(1058),
  2013.

\bibitem[Cha94]{chat94}
Zo{\'e}~Maria Chatzidakis.
\newblock {Some remarks on profinite HNN extensions}.
\newblock {\em {Israel Journal of Mathematics}}, 85(1-3):11--18, 1994.

\bibitem[Hem87]{Hempel87}
John Hempel.
\newblock {Residual finiteness for 3-manifolds}.
\newblock {\em {Combinatorial group theory and topology}}, 111:379--396, 1987.

\bibitem[Hig64]{Hig64}
Graham Higman.
\newblock {Amalgams of $p$-groups}.
\newblock {\em {Journal of Algebra}}, 1(3):301--305, 1964.

\bibitem[Ser03]{SerreTrees}
Jean-Pierre Serre.
\newblock {\em {Trees. Translated from the French original by John Stillwell.
  Corrected 2nd printing of the 1980 English translation}}.
\newblock {Springer Monographs in Mathematics}. {Springer-Verlag, Berlin},
  2003.

\bibitem[Toi13]{Toinet13}
Emmanuel Toinet.
\newblock {Conjugacy $ p $-separability of right-angled Artin groups and
  applications}.
\newblock {\em Groups, Geometry, and Dynamics}, 7(3):751--790, 2013.

\bibitem[Wil17]{Wilkes16b}
Gareth Wilkes.
\newblock {Virtual pro-$p$ properties of 3-manifold groups}.
\newblock {\em {Journal of Group Theory}}, 20(5):999--1023, 2017.

\bibitem[WZ10]{WZ10}
Henry Wilton and Pavel Zalesskii.
\newblock {Profinite properties of graph manifolds}.
\newblock {\em {Geometriae Dedicata}}, 147(1):29--45, 2010.

\end{thebibliography}
\end{document}